\documentclass{article}
\usepackage{amssymb,amsthm,amsmath}
 \newtheorem{thm}{Theorem}[section]
 \newtheorem{cor}[thm]{Corollary}
 \newtheorem{lem}[thm]{Lemma}
 \newtheorem{prop}[thm]{Proposition}
 \theoremstyle{definition}
 
 \theoremstyle{remark}

 \numberwithin{equation}{section}

\def\Pr{\mathbb P}
\def\Ex{\mathbb E}
\def\Var{\mathrm{Var}}
\def\zet{\mathbb Z}

\begin{document}

%-------------------------------------------------------------------------
% editorial commands: to be inserted by the editorial office
%
%\firstpage{1} \volume{228} \Copyrightyear{2004} \DOI{003-0001}
%\seriesextraline{This is the Concrete Title of this Book\br H.E. R and S.T.C. W, Eds.}
%\firstpage{1}

\title{Maximal inequalities for centered 
norms of sums of independent random vectors}
\author{Rafa{\l} Lata{\l}a\thanks{Research partially supported by MNiSW Grant no. N N201 397437.}}
\date{}

%%\thanks{Research partially supported by MNiSW Grant no. N N201 397437.}

%%\subjclass{Primary 60E15; Secondary 60G50,60B11}

%%\keywords{Sums of independent random variables, random vectors, maximal inequality}

\maketitle

\begin{abstract}
Let $X_1,X_2,\ldots,X_n$ be independent random variables and $S_k=\sum_{i=1}^k X_i$. We show that for
any constants $a_k$,
\[
\Pr(\max_{1\leq k\leq n}||S_{k}|-a_{k}|>11t)\leq 30
\max_{1\leq k\leq n}\Pr(||S_{k}|-a_{k}|>t).
\]
We also discuss similar inequalities for sums of Hilbert and Banach space valued random vectors.
\end{abstract}

\section{Introduction and Main Results}

Let $X_1,X_2,\ldots$ be independent random vectors in a separable Banach space $F$.
The L\'evy-Ottaviani maximal inequality (see e.g. Proposition 1.1.1 in \cite{KW}) states that for any $t>0$,
\begin{equation}
\label{L-O}
\Pr\Big(\max_{1\leq k\leq n}\|S_{k}\|>3t\Big)\leq 3\max_{1\leq k\leq n}\Pr(\|S_{k}\|>t),
\end{equation}
where here and in the rest of this note,
\[
S_k=\sum_{i=1}^k X_i \quad \mbox{ for } k=1,2,\ldots.
\]
If, additionally, variables $X_i$ are symmetric then the classical L\'evy inequality gives the sharper bound
\[
\Pr\Big(\max_{1\leq k\leq n}\|S_{k}\|>t\Big)\leq 2\Pr(\|S_{n}\|>t).
\]
Montgomery-Smith \cite{MS} showed that if we replace symmetry assumptions by the identical distribution then
\begin{equation}
\label{M-S}
\Pr\Big(\max_{1\leq k\leq n}\|S_{k}\|>C_1t\Big)\leq C_2\Pr(\|S_{n}\|>t),
\end{equation}
where one may take $C_1=30$ and $C_2=9$.

Maximal inequalities are fundamental tools in the study of convergence of random series and 
limit theorems for sums of independent random vectors (see e.g. \cite{KW} and \cite{LT}).

In some applications one needs to investigate asymptotic behaviour of centered norms of sums, i.e. random variables
of the form $(\|S_n\|-a_n)/b_n$ (cf. \cite{GM}). For such purpose it is natural to ask whether in \eqref{L-O} or \eqref{M-S}
one may replace variables $\|S_k\|$ by $|\|S_k\|-a_k|$. The answer turns out to be positive in the real case.

\begin{thm}
\label{maxreal}
Let $X_{1},X_{2},\ldots,X_{n}$ be independent real r.v.'s. Then
for any numbers $a_{1},a_{2},\ldots,a_{n}$ and $t>0$,
\begin{equation}
\label{real0}
\Pr\Big(\max_{1\leq k\leq n}||S_{k}|-a_{k}|>11t\Big)\leq 30
  \max_{1\leq k\leq n}\Pr(||S_{k}|-a_{k}|>t).
\end{equation}
\end{thm}

\noindent
{\bf Example.} Let $Y_1,Y_2,\ldots$ be i.i.d. r.v.'s such that $\Ex Y_i^2=1$ and $\Var(Y_i^2)<\infty$.
Let $S_k=\sum_{i=1}^k X_i$, where $X_i=e_iY_i$ and $(e_i)$ is an orthonormal system in a Hilbert space 
$\mathcal{H}$; also let $|x|$ denote the norm of a vector $x\in\mathcal{H}$.
Then for $t>0$,
\[
\Pr(||S_k|-\sqrt{k}|\geq t)\leq \Pr(||S_k|^2-k|\geq t\sqrt{k})\leq \frac{\Var(|S_k|^2)}{t^2 k}=
\frac{\Var(Y_1^2)}{t^2}.
\]
On the other hand if we choose $j_0$ such that $2^{j_0/2}\geq t$, then for $n\geq 2^{j_0}$,
\begin{align*}
p_n&:=\Pr\Big(\max_{1\leq k\leq n}||S_k|-\sqrt{k}|\geq t\Big)\geq
\Pr\Big(\max_{2^{j_0}\leq k\leq n}(|S_k|^2-k)\geq 3t\sqrt{k}\Big)
\\
&\geq 
\Pr\bigg(\bigcup_{j_0\leq j\leq \log_2 n}\Big\{|S_{2^j}|^2-2^{j}\geq 3\cdot 2^{j/2}t\Big\}\bigg)
\\
&\geq \Pr\bigg(\bigcup_{j_0+1\leq j\leq \log_2 n}\Big\{2^{-j/2}\sum_{i=2^{j-1}+1}^{2^j}(Y_i^2-1)\geq 6t\Big\}\bigg)
\end{align*}
and $\lim_{n\rightarrow\infty}p_n=1$ for any $t>0$ by the CLT. It is not hard to modify this example in such 
a way that $X_i$ be an i.i.d. sequence. 

\medskip

Hence Theorem \ref{maxreal} does not hold in infinite dimensional Hilbert spaces even if we assume
that $X_i$ are symmetric and identically distributed. However a modification of \eqref{real0} is satisfied
in Hilbert spaces.

\begin{prop}
\label{maxsqr}
Let $X_{1},\ldots,X_{n}$ be independent symmetric r.v.'s with values in a separable
Hilbert space $({\mathcal H},|\ |)$. Then for any sequence of real numbers
$a_{1},\ldots,a_{n}$  and $t\geq 0,$ 
\[
\Pr\Big(\max_{1\leq k\leq n}\big||S_{k}|^{2}-a_{k}\big|\geq 3t\Big)\leq
6\max_{1\leq k\leq n}\Pr\big(\big||S_{k}|^{2}-a_{k}\big|\geq t\big).
\]
\end{prop}

A first consequence of Proposition \ref{maxsqr} is the following Hilbert-space version of \eqref{real0} 
under a regularity assumption on coefficients $(a_k)$.

\begin{cor}
\label{maxrega}
Let $X_{1},\ldots,X_{n}$ be as in Proposition \ref{maxsqr}, $1\leq i\leq n$ and nonnegative real numbers 
$a_{i},\ldots,a_{n}$, $\alpha,$ $\beta$  and $t$ satisfy the condition 
\begin{equation}
\label{coef}
a_k\leq \alpha a_l+\beta t \quad  \mbox{for all } i\leq k,l\leq n. 
\end{equation}
Then
\[
\Pr\Big(\max_{i\leq k\leq n}||S_{k}|-a_{k}|\geq (6\alpha+2\beta+1)t\Big)\leq
6\max_{i\leq k\leq n}\Pr\big(||S_{k}|-a_{k}|\geq t\big).
\]
\end{cor}

In proofs of limit theorems one typically applies maximal inequalities to uniformly estimate  $\|S_k\|$ 
for $cn\leq k\leq n$, where $c$ is some constant. Next two corollaries show that if we restrict
$k$ to such a group of indices then, under i.i.d. and symmetry assumptions, \eqref{real0} holds in Hilbert spaces.

\begin{cor}
\label{max_in_group}
Let $X_{1},X_{2},\ldots,X_n$ be symmetric i.i.d. r.v.'s with values in a separable
Hilbert space $({\mathcal H},|\ |)$. Then for any integer $i$ such that $\frac{n}{2}\leq i\leq n$
and any sequence of positive numbers $a_{i},\ldots,a_{n}$ and $t\geq 0$ we have
\[
\Pr\Big(\max_{i\leq k\leq n}||S_{k}|-a_{k}|\geq 19t\Big)
\leq
6\max_{i\leq k\leq n}\Pr(||S_{k}|-a_{k}|\geq t).
\]
\end{cor}

\begin{proof}
We may obviously assume that
\[
\max_{i\leq k\leq n}\Pr(||S_{k}|-a_{k}|\geq t)\leq \frac{1}{6}.
\]
Observe that for any $k<l$, the random variable $S_{k,l}:=\sum_{i=k}^{l}X_i$ has the same distribution as 
$S_{l-k+1}$.

Take $k,l\in\{i,\ldots,n\}$, then 
\begin{align*}
\Pr(|S_{2k}|\geq 2a_{k}+2t)
&\leq
\Pr(|S_{k}|\geq a_{k}+t)+\Pr(|S_{k+1,2k}|\geq a_{k}+t)
\\
&= 2\Pr(|S_{k}|\geq a_{k}+t)\leq \frac{1}{3}.
\end{align*}
Therefore
\begin{align*}
\Pr(&a_{l}-t\leq |S_{l}|\leq 2a_{k}+2t)
\\
&\geq
 \Pr(a_{l}-t\leq |S_{l}|,\ |S_{l}+S_{l+1,2k}|\leq 2a_{k}+2t,\ |S_{l}-S_{l+1,2k}|\leq 2a_{k}+2t)
\\
&\geq 1-\Pr(|S_{l}|<a_{l}-t)-2\Pr(|S_{2k}|> 2a_{k}+2t)\geq
  1-\frac{1}{6}-\frac{2}{3}>0,
\end{align*}
where in the second inequality we used the symmetry of $X_{i}$. Hence we get $a_{l}\leq 2a_{k}+3t$ and
we may apply Corollary \ref{maxrega} with $\alpha=2$ and $\beta=3$.
\end{proof}

\begin{cor}
Let $X_{1},X_{2},\ldots,X_n$ be as before. Then for any $\frac{n}{2^j}\leq i\leq n$
and any sequence of positive numbers $a_{i},\ldots,a_{n}$ and $t\geq 0$ we have
\[
\Pr\Big(\max_{i\leq k\leq n}||S_{k}|-a_{k}|\geq 19t\Big)
\leq
6j\max_{i\leq k\leq n}\Pr(||S_{k}|-a_{k}|\geq t).
\]
\end{cor}

Corollary \ref{max_in_group} naturally leads to the formulation of the following open question.

\medskip

\noindent
{\bf Question.} Characterize all separable Banach spaces $(E,\|\ \|)$ with the following property.
There exist constants $C_1,C_2<\infty$ such that for any symmetric i.i.d. r.v.'s $X_{1},X_{2},\ldots,X_n$
with values in $E$, any $\frac{n}{2}\leq i\leq n$, any positive constants $a_i,\ldots,a_n$ and $t>0$,
\begin{equation}
\label{max_quest}
\Pr\Big(\max_{i\leq k\leq n}|\|S_{k}\|-a_{k}|\geq C_1t)
\leq
C_2\max_{i\leq k\leq n}\Pr(|\|S_{k}\|-a_{k}|\geq t). 
\end{equation}
In particular does the above inequality hold in $L^p$ with $1<p<\infty$?

\medskip

In the last section of the paper we present an example showing that in a general separable Banach space 
estimate \eqref{max_quest} does not hold.

\section{Proofs}

Below we will use the following notation. By $\tilde{X}_1,\tilde{X}_2,\ldots$ we will
denote the independent copy of the random sequence $X_1,X_2,\ldots$. We put
\[
\tilde{S}_{k}:=\sum_{i=1}^{k}\tilde{X}_{i},\quad
S_{k,n}:=S_{n}-S_{k-1}=\sum_{i=k}^{n}X_{i}.
\]

We start with the following simple lemma.

\begin{lem}
\label{lem1}
Suppose that  real numbers $x,$ $y,$ $a,$ $b$ and $u$ satisfy the conditions
$||x|-a|\leq u,$ $||y|-a|\leq u,$ $||x+s|-b|\leq u,$ $||y+s|-b|\leq u$
and $|x-y|> 2u$. Then $|a-b|\leq 2u$ and $|s|\leq 4u$.
\end{lem}

\begin{proof}
If $a<0$ then $|x|,|y|<u$ and $|x-y|<2u$. So $a\geq 0$ and
in the same way we show that $b\geq 0$. Without loss of generality we may
assume $x<y$, hence $x\in(-a-u,-a+u),$ $y\in(a-u,a+u),$ $x+s\in(-b-u,-b+u),$
$y+s\in(b-u,b+u)$. Thus $2a-2u\leq y-x\leq 2a+2u$ and $2b-2u\leq (y+s)-(x+s)
\leq 2b+2u$ and therefore $|a-b|\leq 2u$. Moreover, $-b+a-2u\leq s\leq -b+a+2u$
and we get $|s|\leq |a-b|+2u\leq 4u$. 
\end{proof}

\begin{proof}[Proof of Theorem \ref{maxreal}]
We may and will assume that
\[
p:=\max_{1\leq k\leq n}\Pr(||S_{k}|-a_{k}|>t)\in (0,1/30).
\]
Let
\[
I_{1}:=\{k\colon\ a_{k}\leq 2t\},\quad I_{2}:=\{k\colon\ \Pr(|S_{k}-\tilde{S}_{k}|>2t)>5p\}
\]
and
\[
I_{3}:=\{1,\ldots,n\}\setminus(I_{1}\cup I_{2}).
\]

First we show that
\begin{equation}
\label{real1}
\Pr\Big(\max_{k\in I_{1}}||S_{k}|-a_{k}|>11t\Big)\leq 3p.
\end{equation}
Indeed, notice that for all $k$, $a_{k}>-t$ (otherwise $p=1$). Therefore by the L\'evy-Ottaviani
inequality \eqref{L-O},
\begin{align*}
\Pr\Big(\max_{k\in I_{1}}||S_{k}|-a_{k}|>11t\Big)
&\leq
\Pr(\max_{k\in I_{1}}|S_{k}|>9t)
\leq 3\max_{k\in I_{1}}\Pr(|S_{k}|>3t)
\\
&\leq
  3\max_{k\in I_{1}}\Pr(||S_{k}|-a_{k}|>t)\leq 3p.
\end{align*}

Next we prove that
\begin{equation}
\label{real2}
\Pr\Big(\max_{k\in I_{2}}||S_{k}|-a_{k}|>11t\Big)\leq 5p.
\end{equation}
Let us take $k\in I_{2}$ and define the following events
\[
A_{1}:=\{|S_{k}-\tilde{S}_{k}|>2t\},\quad A_{2}:=A_{1}\cap \{|S_{k+1,n}|>4t\}
\]
and
\[
B:=\{||S_{k}|-a_{k}|\leq t, ||\tilde{S}_{k}|-a_{k}|\leq t,
||S_{n}|-a_{n}|\leq t, ||\tilde{S}_{k}+S_{k+1,n}|-a_{n}|\leq t\}.
\]
We have $\Pr(A_{1})+\Pr(B)>5p+1-4p>1$, hence $A_{1}\cap B\neq \emptyset$ and by
Lemma \ref{lem1}, $|a_{k}-a_{n}|\leq 2t$. Also by Lemma \ref{lem1}, $A_{2}\cap B= \emptyset$,
hence $\Pr(A_{2})+\Pr(B)\leq 1$. Therefore $5p\Pr(|S_{k+1,n}|>4t)\leq \Pr(A_{2})\leq 4p$.
Thus for all $k\in I_{2}$, $|a_{k}-a_{n}|\leq 2t$ and
$\Pr(|S_{k+1,n}|\leq 4t)\geq 1/5$. Let
\[
\tau:=\inf\{k\in I_{2}:||S_{k}|-a_{k}|>11t \}.
\]
Then
\begin{align*}
\frac{1}{5}\Pr(\tau=k)
&\leq \Pr(\tau=k,|S_{k+1,n}|\leq 4t)
\\
&\leq \Pr(\tau=k,||S_{n}|-a_{n}|>11t-4t-|a_{k}-a_{n}|)
\\
&\leq\Pr(\tau=k,||S_{n}|-a_{n}|>t)
\end{align*}
and
\begin{align*}
\Pr\Big(\max_{k\in I_{2}}||S_{k}|-a_{k}|>11t\Big)
& = \sum_{k\in I_{2}}\Pr(\tau=k)
\leq 5\sum_{k\in I_{2}}\Pr(\tau=k,||S_{n}|-a_{n}|>t)
\\
&\leq 5\Pr(||S_{n}|-a_{n}|>t) \leq 5p.
\end{align*}

Finally we show
\begin{equation}
\label{real3}
\Pr\Big(\max_{k\in I_{3}}||S_{k}|-a_{k}|>11t\Big)\leq 21p.
\end{equation}
To this end take any $k\in I_{3}$ and notice that
\begin{align*}
2\max\{\Pr(|S_{k}-a_{k}|\leq t),&\Pr(|S_{k}+a_{k}|\leq t)\}
\\
&\geq
\Pr(|S_{k}-a_{k}|\leq t)+\Pr(|S_{k}+a_{k}|\leq t)
\\
&\geq \Pr(||S_{k}|-a_{k}|\leq t)\geq 1-p\geq \frac{29}{30}.
\end{align*}
If $|x-a_{k}|\leq t$ and $|y+a_{k}|\leq t$ then $|x-y|\geq 2a_{k}-2t>2t$.
Therefore
\begin{align*}
5p&\geq\Pr(|S_{k}-\tilde{S}_{k}|>2t)
\\
&\geq
 \Pr(|S_{k}-a_{k}|\leq t,|\tilde{S}_{k}+a_{k}|\leq t)+
 \Pr(|S_{k}+a_{k}|\leq t,|\tilde{S}_{k}-a_{k}|\leq t) 
\\
&=2\Pr(|S_{k}-a_{k}|\leq t)\Pr(|S_{k}+a_{k}|\leq t).
\end{align*}
So for any $k\in I_{3}$ we may choose $b_k=\pm a_k$ such that
\[
\Pr(|S_{k}-b_{k}|\leq t)\leq \frac{30}{29}5p\leq 6p.
\]
Therefore
\[
\Pr(|S_{k}+b_{k}|> t)\leq \Pr(||S_{k}|-a_{k}|>t)+\Pr(|S_{k}-b_{k}|\leq t)
\leq 7p
\]
and by the L\'evy-Ottaviani inequality \eqref{L-O},
\begin{align*}
\Pr(\max_{k\in I_{3}}||S_{k}|-a_{k}|>11t)
&\leq \Pr(\max_{k\in I_{3}}|S_{k}+b_{k}|>11t)
\\
&\leq 3\max_{k\in I_{3}}\Pr(|S_{k}+b_{k}|>\frac{11}{3}t)
\leq 21 p.
\end{align*}
This shows (\ref{real3}).

Inequalities (\ref{real1}), (\ref{real2}) and (\ref{real3}) imply
(\ref{real0}). 
\end{proof}

\begin{proof}[Proof of Proposition \ref{maxsqr}]
It is enough to consider the case when
\[
p:=\max_{1\leq k\leq n}\Pr\big(\big||S_{k}|^{2}-a_{k}\big|\geq t\big)<\frac{1}{6}.
\]

Notice that
\[
\Pr\big(\big||S_{n}|^{2}-|S_{k}|^{2}-(a_{n}-a_{k})\big|\geq 2t\big)
\leq
\Pr\big(\big||S_{n}|^{2}-a_{n}\big|\geq t\big)+\Pr\big(\big||S_{k}|^{2}-a_{k}\big|\geq t\big)
%%\leq 2p.
\]
Therefore
\[
\Pr\big(\big||S_{k+1,n}|^{2}+2\langle S_{k},S_{k+1,n}\rangle-(a_{n}-a_{k})\big|\geq 2t)
\leq 2p\]
and by the symmetry
\[
\Pr\big(\big||S_{k+1,n}|^{2}-2\langle S_{k},S_{k+1,n}\rangle-(a_{n}-a_{k})\big|\geq 2t\big)
\leq 2p.
\]
Thus by the triangle inequality
\[
\Pr\big(\big||S_{k+1,n}|^{2}-(a_{n}-a_{k})\big|\geq 2t\big)\leq 4p.
\]

Now let $x\in {\mathcal H}$ be such that $||x|^{2}-a_{k}|\geq 3t$ then by the
triangle inequality and symmetry
\begin{align*}
1-4p
&\leq \Pr\big(\big||x|^{2}+|S_{k+1,n}|^{2}-a_{n}\big|\geq t\big)
\\
&\leq \Pr\big(\big||x|^{2}+|S_{k+1,n}|^{2}+2\langle x,S_{k+1,n}\rangle-a_{n}\big|\geq t\big)
\\  
&\phantom{\leq}+
\Pr\big(\big||x|^{2}+|S_{k+1,n}|^{2}-2\langle x,S_{k+1,n}\rangle -a_{n}\big|\geq t\big)
\\
&=2\Pr\big(\big||x|^{2}+|S_{k+1,n}|^{2}+2\langle x,S_{k+1,n}\rangle-a_{n}\big|\geq t\big)
\\
&=2\Pr\big(\big||x+S_{k+1,n}|^{2}-a_{n}\big|\geq t\big).
\end{align*}
So for any $x\in {\mathcal H}$ and $k=1,2,\ldots,n$,
\begin{equation}
\label{impl1}
\big||x|^{2}-a_{k}\big|\geq 3t
\ \Rightarrow\
\Pr\big(\big||x+S_{k+1,n}|^{2}-a_{n}\big|\geq t\big) \geq \frac{1}{2}(1-4p)\geq \frac{1}{6}.
\end{equation}

Now let
\[
\tau:=\inf\big\{k\leq n\colon\ \big||S_{k}|^{2}-a_{k}\big|\geq 3t\big\},
\]
then since $\{\tau =k\}\in \sigma(X_{1},\ldots,X_{k})$ we get by \eqref{impl1},
\[
\Pr\big(\tau=k,\big||S_{n}|^{2}-a_{n}\big|\geq t\big)\geq \frac{1}{6}\Pr(\tau=k).
\]
Hence
\[
\Pr\big(\big||S_{n}|^{2}-a_{n}\big|\geq t\big)
\geq \frac{1}{6}\sum_{k=1}^{n}\Pr(\tau=k)=
\frac{1}{6}\Pr\Big(\max_{1\leq k\leq n} \big||S_{k}|^{2}-a_{k}\big|\geq 3t\Big)
\]
and the proposition follows.
\end{proof}

\begin{proof}[Proof of Corollary \ref{maxrega}]
We may consider variables $S_i,X_{i+1},\ldots,X_n$ instead of $X_1,\ldots,X_n$ and assume that $i=1$.
Let $a:=\min_{1\leq k\leq n} a_{k}$. We will analyze two cases.

\medskip

\noindent
{\bf Case 1.} $a\leq 3t$. Then by (\ref{coef}) we get $a_{k}\leq (3\alpha+\beta)t$ for all $k$. Thus
by the L\'evy inequality,
\begin{align*}
\Pr\Big(\max_{k}||S_{k}|-a_{k}|\geq (6\alpha+2\beta+1)t\Big)&\leq
\Pr\Big(\max_{k}|S_{k}|\geq (3\alpha+\beta+1)t\Big)
\\
&\leq
2\Pr(|S_{n}|\geq (3\alpha+\beta+1)t)
\\
&\leq 2\Pr(||S_{n}|-a_{n}|\geq t)
\\
&\leq 
2\max_{k}\Pr(||S_{k}|-a_{k}|\geq t).
\end{align*}

\medskip

\noindent
{\bf Case 2.} $a\geq 3t$. Notice first that for any $s>0$ we have
\begin{equation}
\label{incl}
\{||S_{k}|^{2}-a_{k}^{2}|\geq s(2a_{k}+s)\}\subset
\{||S_{k}|-a_{k}|\geq s\}\subset
\{||S_{k}|^{2}-a_{k}^{2}|\geq sa_{k}\}.
\end{equation}
Indeed, the last inclusion follows since
$||S_{k}|^{2}-a_{k}^{2}|=(|S_{k}|+a_{k})||S_{k}|-a_{k}|\geq a_{k}||S_{k}|-a_{k}|$. 
To see the first inclusion in \eqref{incl} observe that
\begin{align*}
\{||S_{k}|^{2}-a_{k}^{2}|\geq s(2a_{k}+s)\}
&\subset
\{||S_{k}|-a_{k}|\geq s\}\cup\{|S_{k}|+a_{k}\geq 2a_{k}+s\}
\\
&\subset \{||S_{k}|-a_{k}|\geq s\}.
\end{align*}

Now by \eqref{incl} we get
\begin{align*}
\Pr\Big(\max_{k}|&|S_{k}|-a_{k}|\geq (6\alpha+2\beta+1)t\Big)
\\
&\leq 
\Pr\Big(\max_{k}||S_{k}|^{2}-a_{k}^{2}|\geq (6\alpha+2\beta+1)at\Big).
\end{align*}
Hence by Proposition \ref{maxsqr},
\[
\Pr\Big(\max_{k}||S_{k}|-a_{k}|\geq (6\alpha+2\beta+1)t\Big)
\leq
6\max_{k}\Pr\big(||S_{k}|^{2}-a_{k}^{2}|\geq \frac{1}{3}(6\alpha+2\beta+1)at\big).
\]
But $\frac{1}{3}(6\alpha+2\beta+1)a\geq 2(\alpha a+\beta t)+t\geq 2a_{k}+t$ for all $k$ by
\eqref{coef}. Therefore by \eqref{incl},
\begin{align*}
\Pr\Big(\max_{k}||S_{k}|-a_{k}|\geq (6\alpha+2\beta+1)t\Big)
&\leq 6\max_{k}\Pr(||S_{k}|^{2}-a_{k}^{2}|\geq t(2a_{k}+t))
\\
&\leq 6\max_{k}\Pr(||S_{k}|-a_{k}|\geq t).
\end{align*}
\end{proof}

\section{Example}

Let us fix a positive integer $n$ and put
\[
I_n=\Big\{j\in \zet\colon\ \frac{n}{2}\leq j\leq n\Big\}.
\]
Let $t_j=\frac{n^2+j}{j}$ for $j=1,2,\ldots, n$, then
\begin{equation}
\label{propt}
jt_j=n^2+j \quad \mbox{ and }\quad (j-1)t_j\leq n^2 \quad \mbox{ for }j\in I_n.
\end{equation}

Let $N$ be a large integer (to be fixed later) and let $F$ be the space of all double-indexed sequences
$a=(a_{i,j})_{0\leq i\leq N,j\in I_n}$ with the norm
\[
\Big\|(a_{i,j})_{0\leq i\leq N,j\in I_n}\Big\|=
\max_{j\in I_n}\bigg(|a_{0,j}|+t_j\sum_{1\leq i_1<i_2<\ldots<i_j\leq N}\sum_{s=1}^j|a_{i_s,j}|
\bigg).
\]
Let $(e_{i,j})$ be a standard basis of $F$, so that $(a_{i,j})=\sum_{i,j}a_{i,j}e_{i,j}$. 

Define random vectors $X_1,X_2,\ldots,X_n$  by the formula
\[
X_l=\sum_{j\in I_n}(Y_{l,j}e_{0,j}+R_{l,j}e_{N_l,j}),
\]
where $(Y_{l,j},R_{l,j})_{l\leq n,j\in I_n}$ and $(N_l)_{l\leq n}$ are independent r.v's,
$\Pr(R_{l,j}=\pm 1)=1/2$, $Y_{l,j}$ are symmetric $\Pr(|Y_{l,j}|=\frac{1}{2n})=1-\Pr(Y_{k,j}=0)=p_n$
(with $p_n$ a small positive number to be specified later) and $N_l$ are uniformly sampled from the set
$\{1,\ldots,N\}$.

Obviously $X_1,X_2,\ldots,X_n$ are i.i.d. and symmetric. As usual we set $S_k=X_1+X_2+\ldots+X_k$.
Let 
\[
A=\{ N_1,N_2,\ldots,N_n \mbox{ are pairwise distinct}\}.
\]
Notice that $\Pr(A^c)\rightarrow 0$ when $N\rightarrow\infty$.  On the set $A$ we have
for $k\leq n$,
\[
\|S_k\|=\max_{j\in I_n}\Big(\Big|\sum_{l=1}^k Y_{l,j}\Big|+t_j\min\{k,j\}\Big).
\]
For $j>k$ we have by \eqref{propt},
\[
\Big|\sum_{l=1}^k Y_{l,j}\Big|+t_j\min\{k,j\}< 1+t_j(j-1)\leq n^2+1,
\]
hence on the set $A$, for $k\in I_n$ we get
\[
\|S_k\|=\max_{j\in I_n,j\leq k}\Big(\Big|\sum_{l=1}^k Y_{l,j}\Big|+n^2+j\Big)
=\Big|\sum_{l=1}^k Y_{l,k}\Big|+n^2+k.
\]

Take $0<t< \frac{1}{2nC_1}$ then for $k\in I_n$,
\[
\Pr(|\|S_k\|-(n^2+k)|\geq t)\leq \Pr(A^c)+\Pr\Big(\sum_{l=1}^kY_{l,k}\neq 0\Big)\leq \Pr(A^c)+kp_n
\]
and
\[
\Pr\Big(\max_{k\in I_n}|\|S_k\|-(n^2+k)|\geq tC_1\Big)\geq
\Pr\Big(\max_{k\in I_n}\Big|\sum_{l=1}^kY_{l,k}\Big|\neq 0\Big)-\Pr(A^c).
\]
The last number is of order $n^2p_n$ if $N$ is large and $p_n$ is small. This shows that
if \eqref{max_quest} holds for $i=\lceil n/2\rceil $ in $F$ then $C_2$ must be of order $n$. So \eqref{max_quest} cannot
hold with absolute constants $C_1$ and $C_2$ in (infinite dimensional) separable Banach spaces.

\noindent
Institute of Mathematics\\
University of Warsaw\\ 
Banacha 2\\
02-097 Warszawa\\
Poland\\
{\tt rlatala@mimuw.edu.pl}

\end{document}